\documentclass[apaper,12pt,reqno]{amsart}
\usepackage{amssymb}
\usepackage{amsmath}
\usepackage{amsbsy}
\usepackage{amsthm}
\usepackage{mathtools}

\usepackage{bm}
\usepackage{MnSymbol}

\newtheorem{thm}{Theorem}[section]
\newtheorem{cor}[thm]{Corollary}
\newtheorem{lem}[thm]{Lemma}
\newtheorem{prop}[thm]{Proposition}
\theoremstyle{definition}
\newtheorem{defn}[thm]{Definition}
\theoremstyle{remark}

\theoremstyle{conjecture}
\newtheorem{conjecture}{Conjecture}

\numberwithin{equation}{section}

\newcommand{\Int}{\mathbf{Int}}
\newcommand{\Intm}{\mathbf{Int}^{\square}}
\newcommand{\IntmExt}[1]{\Intm_{\textbf{#1}}}

\newcommand{\kfourGrz}{\mathbf{K4.Grz}}
\newcommand{\mHC}{\mathbf{mHC}}
\newcommand{\E}{\mathbf{E}}
\newcommand{\KM}{\mathbf{KM}}

\newcommand{\Kuz}{\mathbf{Kuz}}
\newcommand{\Kuzs}{\mathbf{Kuz}^{\ast}}

\newcommand{\C}[1]{\text{C}_{\text{#1}}}

\newcommand{\Lana}{\mathfrak{L}_{a}}
\newcommand{\Lanm}{\mathfrak{L}_{m}}
\newcommand{\Lanb}{\mathfrak{L}_{b}}

\newcommand{\fA}{\mbox{$\mathfrak{A}$}}
\newcommand{\fB}{\mbox{$\mathfrak{B}$}}

\newcommand{\Fr}{\mathfrak{F}}

\newcommand{\tr}{\textbf{\textit{t}}}
\newcommand{\spl}{\textbf{\textit{s}}}

\newcommand{\Var}{\mbox{$\textit{\textbf{Var}}$}}

\newcommand{\ra}{\rightarrow}

\newcommand{\on}{\mbox{$\mathbf{1}$}}
\newcommand{\ze}{\mbox{$\mathbf{0}$}}
\newcommand{\true}{\bm{\top}}

\begin{document}
\date{}

\title[On Some Syntactic Properties]{On Some Syntactic Properties of the Modalized Heyting Calculus}
\author{Alexei Muravitsky}

\maketitle

\begin{abstract}
We show that the modalized Heyting calculus~\cite{esa06} admits a normal axiomatization. Then we prove that in this calculus the inference rule $\square\alpha/\alpha$ is admissible (Proposition~\ref{P:weakening-2}), but the rule $\square\alpha\rightarrow\alpha/\alpha$ is not (Proposition \ref{P:equivalences}). Finally, we show that this calculus and intuitionistic propositional calculus are assertorically equipollent, which leads to a variant of limited separation property for the modalized Heyting calculus.
\end{abstract}

\section{Introduction}

The modalized intuitionistic calculus $\mHC$ was introduced by Leo Esakia~\cite{esa06} as a weakening of the proof-intuitionistic logic, nowadays known as $\KM$; see, e.g.,~\cite{mur14b}.
In Section~\ref{S:languages}, we will give another axiomatization of $\mHC$ and call it $\E$ (after Leo Esakia). The main goal of this reformulation is that $\E$ is a normal axiomatic system (in the sense of~\cite{ho73}, p. 75), while $\mHC$ is not. The last circumstance leads to the fact that the calculus $\mHC$ does not possess the separation property; for $\E$ the question is open, though a limited version of it is presented in Section~\ref{S:equipollent}. The present work has been done in direction of (and with hopes for) answering this question in the affirmative. Thus we will be focusing on syntactic, that is proof-theoretic, properties of $\E$ in the Hilbert-style framework.

\section{Languages and systems}\label{S:languages}
We fix a sentential language, $\Lana$, based on a countable set $\Var$ of sentential variables and the assertoric logical connectives: $\wedge$ (conjunction), $\vee$ (disjunction), $\ra$ (conditional, or implication, or entailment), and $\neg$ (negation). Unspecified variables of $\Var$ will be denoted by letters $p, q,r,\ldots$ and unspecified $\Lana$-\textit{formulas} by letters $A, B, C\ldots$. By adding a unary connective $\square$ (modality) to $\Lana$, we obtain language $\Lanm$, unspecified formulas of which ($\Lanm$-\textit{formulas}) will be denoted by letters $\alpha, \beta, \gamma\ldots$ Formulas of the form $\square\alpha$ are called $\square$-\textit{formulas}. For a fixed variable $p\in\Var$, we denote
\[
\true:=p\ra p.
\]

In Section~\ref{S:equipollent} we will be using the usual operation of replacement of a subformula $\alpha$ of formula $\gamma$ by a formula $\beta$, denoting this operation by
\[
\gamma[\alpha:\beta].
\]
In a natural way, this operation is extended to multiple simultaneous replacement.

From an algebraic viewpoint, each of $\Lana$ and $\Lanm$ defines a similarity type and so does any of their reductions. For any of these similarity types (or languages) one can define a \textit{formula algebra}, $\Fr$. Given a formula algebra $\Fr$, a \textit{substitution} is a homomorphism of $\Fr$ into
$\mathfrak{F}$.\\

Next we introduce main calculi we will be dealing with. All these calculi have one and the same set of \textit{inference rules} --- (uniform) substitution and modus ponens.

Intuitionistic propositional calculus $\Int$ is defined by the following axioms divided into the four groups:
\begin{equation}\label{E:Int-axioms}
\begin{array}{cl}
(\textbf{i}) &p\ra(q\ra p),~(p\ra(q\ra r))\ra((p\ra q)\ra(p\ra r))\\
(\textbf{c}) &(p\wedge q)\ra p,~(p\wedge q)\ra q,~
p\ra(q\ra (p\wedge q))\\
(\textbf{d}) &p\ra(p\vee q),~p\ra(q\vee p),~
(p\ra r)\ra((q\ra r)\ra((p\vee q)\ra r))\\
(\textbf{n}) &(p\ra q)\ra((p\ra\neg q)\ra\neg p),~
p\ra(\neg p\ra q),
\end{array}
\end{equation}
where $p,q, r$ are three fixed distinct variables of $\Var$.\footnote{
	These axioms are specifications of the axiom schemata from~\cite{kle52}, \S~19.}

We formulate the modalized Heyting calculus $\E$ by adding to the axioms (\ref{E:Int-axioms}) the following group of formulas:
\[
\begin{array}{cl}
(\textbf{m}) &\square(p\ra q)\ra(\square p\ra\square q),~p\ra\square p,~\square p\ra(((q\ra p)\ra q)\ra q).\tag{\textbf{m}-\textit{axioms}}\label{m-axioms}
\end{array}
\]

Next we define the calculi which will play merely auxiliary role in our discussion. The common framework for these calculi is $\Int$ formulated in the language $\Lanm$, which we denote by $\Intm$. Also, we define:
\begin{itemize}
	\item $\Kuz:=\Intm + \square p\ra(((q\ra p)\ra q)\ra q)$ (where the last formula is the only \textit{modal axiom} of $\Kuz$);
	\item $\Kuzs:=\Intm + \square p\ra(q\vee(q\ra p))$ (with the last formula only \textit{modal axiom});
	\item $\mHC:=\Kuzs+\square(p\ra q)\ra(\square p\ra\square q)+p\ra\square p$; we divide the modal axioms in two groups:
	\[
	\begin{array}{cl}
	(\textbf{m}_{1}) &\square(p\ra q)\ra(\square p\ra\square q),~p\ra\square p,\\ 
	(\textbf{m}_{2}) &\square p\ra(q\vee(q\ra p));
	\end{array}
	\] 
	\item $\KM:=\mHC+(\square p\ra p)\ra p$.
	We note that the first axiom of $(\textbf{m}_{1})$ is redundant; cf.~\cite{km86}, p. 88.	
\end{itemize}

We note that $\E$ differs from $\mHC$ in that the last \textbf{m}-axiom above is replaced with
$\square p\ra(q\vee(q\ra p))$ and that
\[
\E=\Kuz+\square(p\ra q)\ra(\square p\ra\square q)+p\ra\square p
\]

As we will show in Section~\ref{S:deducibilities}, the calculi $\E$ and $\mHC$ generate one and the same logic, that is the same set of derivable formulas.

The following interconnection between $\Intm$ and $\Int$ is almost obvious.
\begin{prop}[cf.~\cite{mur16a}, Proposition 2.4]\label{P:Int-Int-square-interconnection}
For any $\Lanm$-formula $\alpha$, if $\Intm\vdash\alpha$, then there is an $\Lana$-formula $A$ such that $\alpha$ is obtained by substitution from $A$ and $\Int\vdash A$, and conversely. 
\end{prop}

Given a calculus $\text{C}$ and formulas $\alpha_1,\ldots,\alpha_n,\beta$, by
\[
\text{C}+\alpha_1,\ldots,\alpha_n\Vdash \beta
\]
we mean such a deducibility where substitution can be applied only to formulas that are derivable in \text{C}.
We call such a derivation an \text{C}-derivation of $\beta$ from $\alpha_1,\ldots,\alpha_n$ \textit{without substitution} (\textit{w. s.}). For derivations with unrestricted use of substitution, we employ a conventional notation,
\[
\text{C}+\alpha_1,\ldots,\alpha_n\vdash \beta.
\]
It is obvious that both relations $\Vdash$ and $\vdash$ are transitive.\footnote{It should be clear that this restriction on the substitution rule is imposed for the purpose of the use of deduction theorem.} Also, to indicate a fragment of \text{C}, which can be associated with the groups $(\textbf{i})-(\textbf{m})$, we use notation $\C{i}$, $\C{ic}$, etc. 

To illustrate, how we are going to use this notation, we prove that
\begin{equation}\label{E:Int-deducibility}
\Int_{\textbf{icd}}\vdash((q\vee(q\ra p))\ra p)\ra(q\vee(q\ra p)).
\end{equation}

To prepare application of deduction theorem, we prove
that
\[
\Int_{\textbf{icd}}+(q\vee(q\ra p))\ra p\Vdash q\vee(q\ra p).
\]

Indeed, we have:
\[
\begin{array}{ll}
(1) &(q\vee(q\ra p))\ra p\quad(\text{premise})\\
(2) &(q\ra p)\wedge((q\ra p)\ra p)
\quad(\text{from (1) by $\Int_{\textbf{icd}}$-derivation w.s.})\\
(3) &(q\ra p)\wedge p\quad(\text{from (2) by $\Int_{\textbf{ic}}$-derivation w.s.})\\
(4) &q\ra p\quad(\text{from (3) by $\Int_{\textbf{ic}}$-derivation w.s.})\\
(5) &q\vee(q\ra p)\quad(\text{from (4) by $\Int_{\textbf{id}}$-derivation w.s.})
\end{array}
\]

In Section~\ref{S:box-alpha-infers-alpha} we will introduce one more axiomatic system which will play a ``supporting'' role for $\E$.

\section{Some deducibilities}\label{S:deducibilities}
\begin{prop}\label{P:deductive-equivalence}
The following deducibilities hold$\,:$
{\em\[
	\begin{array}{cl}
	(\text{a}) &\Kuz_{\textbf{icdm}}
	\vdash \square p\ra(q\vee(q\ra p))\\
	(\text{b}) &\Kuzs_{\textbf{icdm}}
	\vdash \square p\ra(((q\ra p)\ra q)\ra q).
	\end{array}
	\]}
\end{prop}
\begin{proof}
	To prove (a), we show that
	\begin{equation}\label{E:mHC-deducibility}
	\Kuz_{\textbf{icd}}+\square p\Vdash q\vee(q\ra p).
	\end{equation}
	Indeed, let us denote
	\[
	A:=q\vee(q\ra p).
	\]
	Then, we obtain:
	\[
	\begin{array}{cl}
	(1) &\square p\quad(\text{premise})\\
	(2) & \square p\ra(((A\ra p)\ra A)\ra A)\quad(\text{axiom instance})\\
	(3) &((A\ra p)\ra A)\ra A\quad(\text{from (1) \& (2) by modus ponens})\\
	(4) &(((q\vee(q\ra p))\ra p)\ra(q\vee(q\ra p)))\ra(q\vee(q\ra p))\quad(\text{the same as (3)})\\
	(5) &((q\vee(q\ra p))\ra p)\ra(q\vee(q\ra p))\quad(\text{deducibility (\ref{E:Int-deducibility})})\\
	(6) &q\vee(q\ra p)\quad(\text{from (5) \& (4) by modus ponens})
	\end{array}
	\]
	\
	
	Next we prove that
	\[
	\Kuzs_{\textbf{icd}}+\square p,(q\ra p)\ra q\Vdash q.	
	\]
	Indeed, we have:
	\[
	\begin{array}{cl}
	(1) &\square p,~(q\ra p)\ra q\quad(\text{premises})\\
	(2) &\square p\ra(q\vee(q\ra p))\quad(\text{axiom instance})\\
	(3) &q\vee(q\ra p)\quad(\text{from (1) \& (2) by modus ponens})\\
	(4) &q\ra q\quad(\text{derivable in $\Int_{\textbf{i}}$})\\
	(5) &(q\ra q)\wedge((q\ra p)\ra q)\quad(\text{from (1) \& (4) by $\Int_{\textbf{ic}}$-derivation w.s.})\\
	(6) &(q\vee(q\ra p))\ra q\quad(\text{from (5) by $\Int_{\textbf{id}}$-derivation w.s.})\\
	(7) &q\quad(\text{from (3) \& (6) by modus ponens})
	\end{array}
	\]
\end{proof}
\begin{cor}\label{C:equivalences}
For any formula $\alpha$, the following equivalences hold$\,:$
{\em
	\[
	\begin{array}{cl}
	(\text{a}) &\Kuz\vdash\alpha\Longleftrightarrow
	\Kuzs\vdash\alpha\\
	(\text{b}) &\E\vdash\alpha\Longleftrightarrow
	\mHC\vdash\alpha.
	\end{array}
	\]}
\end{cor}

Following terminology in~\cite{ho73}, p. 75,  $\E$ is a normal axiomatic system for $\mHC$. Usually, normalization is the first step toward obtaining the separation property, though this property can be formulated for non-normal calculi as well. As we will see in the next section, the separation property for $\mHC$ does not hold.

\begin{conjecture}\label{conjecture-1}
The calculus $\E$ possesses the separation property$;$ that is, any formula derivable in $\E$ is also derivable by using only axioms of the group {\em 
(\textbf{i})} and those ones in the groups {\em 
$(\textbf{c})-(\textbf{m})$} which correspond to the logical connectives actually appearing in the formula.
\end{conjecture}

\section{Algebraic background}
Below we consider Heyting algebras in the signature: $\wedge$ (greatest lower bound), $\vee$ (least upper bound), $\ra$ (relative pseudocomplementation), $\neg$
(pseudocomplementation), and $\on$ (unit), as well as their expansions by a unary operation $\square$ (modality). We call the latter algebras \textit{$\square$-enhanced Heyting algebras}.
\begin{defn}[modal Heyting algebra, $\Kuz$-algebra, $\E$-algebra]\label{D:E-algebra}
A $\square$-enhanced Heyting algebra is a modal Heyting algebra if the following identities hold:
\[
\begin{array}{cl}
(\text{a}) &\square\on=\on\\
(\text{b}) &\square(x\wedge y)=\square x\wedge\square y.
\end{array}
\]
The latter algebra is a $\Kuz$-algebra if in addition the next identity is valid:
\[
\begin{array}{cl}
\hspace{-0.15in}(\text{c}) &\square x\le y\vee(y\ra x).
\end{array}
\]	
And the latter in turn is an $\E$-algebra if in addition to $(\text{a})-(\text{c})$ the following identity is true as well:
\[
\begin{array}{cl}
\hspace{-0.8in}(\text{d}) &x\le \square x.
\end{array}
\]
\end{defn}

As usual, we employ the notation
\[
\fA\models\alpha
\]
to indicate that a formula $\alpha$ is valid (in a usual sense) in an algebra $\fA$ (for all types of algebras used in this paper).
\begin{prop}\label{P:semantic-equivalence}
	For any formula $\alpha$,
	{\em\[
		\begin{array}{cl}
		(\text{a}) &\Kuz\vdash\alpha\Longleftrightarrow\Kuzs\vdash\alpha
		\Longleftrightarrow\fA\models\alpha,~\text{for any $\Kuz$-algebra $\fA$;}\\
		(\text{b}) &\E\vdash\alpha\Longleftrightarrow\mHC\vdash\alpha
		\Longleftrightarrow\fA\models\alpha,~\text{for any $\E$-algebra $\fA$}. 
		\end{array}
		\]}
\end{prop}

Now we are ready to demonstrate that the separation property (as it is expressed in Conjecture~\ref{conjecture-1}) does not hold  for $\mHC$. 

Indeed, let us take the formula $\alpha_0=\square p\ra(((q\ra p)\ra q)\ra q)$. This formula contains only two connectives and it is not derivable in the calculus
 $\mHC_{\textbf{im}}$ based on the axioms of the groups (\textbf{i}) and $(\textbf{m}_{1})$. The last claim becomes clear if we consider a 3-element $\square$-enhanced Heyting algebra with $\square x=\on$, since the above formula is invalid in this algebra but all (\textbf{i})-axioms and $(\textbf{m}_{1})$-axioms are. Hence  $\mHC_{\textbf{im}_{1}}\not\vdash\alpha_0$.
 However, according to Proposition~\ref{P:deductive-equivalence}, this formula is derivable in $\mHC$.
 
 \section{Admissibility of the rule $\square\alpha/\alpha$}\label{S:box-alpha-infers-alpha}
 
 To explain the task of this section we need to introduce another player --- logic system $\kfourGrz$ defined in~\cite{esa06}.
\begin{align*}
\kfourGrz:=\Intm+\neg\neg p\ra p+
\square(p\ra q)\ra(\square p\ra\square q)+
\square p\ra\square\square p~+\\
\square(\square(p\ra\square p)\ra p)\ra\square p +\frac{\alpha}{\square\alpha}.
\end{align*} 
We aim to prove that the rule $\frac{\square\alpha}{\alpha}$ is admissible in both
$\kfourGrz$ and $\mHC$, as well as, according to Proposition~\ref{P:semantic-equivalence}, in $\E$.

Since we shall work with the algebraic semantics of $\kfourGrz$, we start with it.
\begin{defn}[$\kfourGrz$-algebra]
Let $\fA=(\mathcal{A},\wedge,\vee,\neg,\on,\square)$ be a Boolean algebra with a unary operation $\square$.
$\fA$ is a $\kfourGrz$-algebra if it is a modal algebra
(that is the identities (a) and (b) of Definition~\ref{D:E-algebra} hold), in which the following identities are valid:
\[
\begin{array}{cl}
(\text{a}^{\ast}) &\square x\le\square\square x\\
(\text{b}^{\ast}) &\square(\neg\square(\neg x\vee\square x)\vee x)\le\square x.
\end{array}
\]
\end{defn}

It is obvious that
\begin{equation}
\kfourGrz\vdash\alpha\Longleftrightarrow\fA\models\alpha,~\text{for any $\kfourGrz$-algebra $\fA$}.
\end{equation}
\begin{defn}[doubling, doubleton]\label{D:doubling}\footnote{This construction was for the first time introduced in~\cite{km80}, p. 216, for a similar purpose we employ it here. In~\cite{mur83} the present construction had slightly been generalized, which was
		later used~\cite{mur85} to prove a property similar to Proposition~\ref{P:continuum}.}
	Let $\fA$ be an algebra of similarity type $\langle\wedge,\vee,\neg,\on,\square\rangle$ and $\mathbf{B}_2$ be a 2-element Boolean algebra. A doubleton of $\fA$ is an algebra $\fB$ of type $\langle\wedge,\vee,\neg,\on,\square\rangle$ such that
	$|\fB|=|\fA|\times|\mathbf{B}_2|$, the operations $\wedge,\vee,\ra,\neg$ are defined as in direct product, and $\square(x,y)=(\square x,z)$, where $z=\on$ if, and only if, $x=\on$.
\end{defn}

Working with expansions of Boolean algebras, we will be using the following notation:
\[
x\Rightarrow y:=\neg x\vee y.
\]
Thus the above condition $(\text{b}^{\ast})$ can be rewritten as follows:
\[
\square(\square(x\Rightarrow\square x)\Rightarrow x)\le\square x.
\]
\begin{prop}\label{P:doubling}
The variety of $\kfourGrz$-algebras is closed under doubling.
\end{prop}
\begin{proof}
Let $\fB$ be the doubleton of a $\kfourGrz$-algebra $\fA$. Let us take two elements, $\bar{x}=(x_1,x_2)$ and $\bar{y}=(y_1,y_2)$, of $|\fB|$. First we notice that
\[
\square(\on,\on)=(\on,\on).
\]

Next we show that
\begin{equation*}\label{E:K-axiom}
\square(\bar{x}\wedge\bar{y})=(\square\bar{x}
\wedge\square\bar{y}),
\end{equation*}
that is
\begin{equation*}\label{E:K-axiom-1}
\square(x_1\wedge y_1,x_2\wedge y_2)=\square (x_1,x_2)\wedge\square(y_1, y_2).
\end{equation*}

Let us denote
\begin{align*}
\square(x_1\wedge y_1,x_2\wedge y_2) &=(\square(x_1\wedge y_1),z_1)\\
\square (x_1,x_2) &=(\square x_1,z_2)\\
\square(y_1, y_2) &=(\square y_1,z_3).
\end{align*}
Thus we have to show that
\begin{equation}\label{E:z-equality}
z_1= z_2\wedge z_3.
\end{equation}

If $x_1\wedge y_1\neq\on$, then $z_1=\ze$ and either
$z_2=\ze$ or $z_3=\ze$. Thus we have \eqref{E:z-equality} true. If $x_1\wedge y_1=\on$, then both $x_1=\on$ and $y_1=\on$. Therefore, $z_1=z_2=z_3=\on$ and hence \eqref{E:z-equality} is true again.

Next we prove that
\[
\square\bar{x}\le\square\square\bar{x},
\]
that is
\[
\square(x_1,x_2)=(\square x_1,z_2)\le\square(\square x_1,z_2).
\]
For this, denoting
\[
\square(\square x_1,z_2)=(\square\square x_1,z_4),
\]
we have to show that
\[
z_2\le z_4.
\]
We have to consider the two cases: $x_1\ne\on$ and $x_1=\on$. In case $x_1\ne\on$, $z_2=\ze$. In case $x_1=\on$, $\square x_1=\on$ and hence $z_4=\on$.

Thus it remains to check that
\[
\square(\square(\bar{x}\Rightarrow\square\bar{x})\Rightarrow\bar{x})\le\square\bar{x},
\]
that is
\[
\square(\square((x_1,x_2)\Rightarrow\square(x_1,x_2))
\Rightarrow(x_1,x_2))\le\square(x_1,x_2).
\]
In terms of notation introduced above, we have to show that
\[
\square(\square(x_1\Rightarrow\square x_1,x_2\Rightarrow z_1)\Rightarrow(x_1,x_2))\le(\square x_1,z_1).
\]
We denote:
\begin{align*}
\square(x_1\Rightarrow\square x_1,x_2\Rightarrow z_1) &=(\square(x_1\Rightarrow\square x_1),z_5)\\
\square(\square(x_1\Rightarrow\square x_1)\Rightarrow x_1,z_5\Rightarrow x_2)
&=(\square(\square(x_1\Rightarrow\square x_1)\Rightarrow x_1),z_6).
\end{align*}
To complete the proof we need to show that
\[
z_6\le z_1.
\]
Indeed, assume first that $\square(x_1\Rightarrow\square x_1)\le x_1$. Then $\square x_1=\on$ and hence $x_1=\on$. The latter means that $z_1=\on$. On the other hand, if $\square(x_1\Rightarrow\square x_1)\not\le x_1$,
then $z_6=\ze$.
\end{proof}
\begin{cor}\label{C:weakening-1}
The inference rule $\square\alpha/\alpha$ {\em (}weakening{\em)} is admissible in $\kfourGrz$.
\end{cor}
\begin{proof}
	Suppose $\kfourGrz\not\vdash\alpha$. This means that $\alpha$ can be refuted in some $\kfourGrz$-algebra $\fA$. Suppose a refuting valuation is $p\mapsto a,\dots,q\mapsto b$, where $p,\dots,q$ are all variables occurring in $\alpha$. Let an algebra $\fB$ be obtained by the doubling of $\fA$. According to Proposition~\ref{P:doubling}, $\fB$ is also a $\kfourGrz$-algebra. It is clear that $\square\alpha$ is refuted in $\fB$ by the valuation $p\mapsto (a,\on),\dots,q\mapsto (b,\on)$. Hence $\kfourGrz\not\vdash\square\alpha$.
\end{proof}

The transfer of the weakening rule from $\kfourGrz$ onto $\mHC$, and thereby (Proposition~\ref{P:semantic-equivalence}) onto $\E$,
can be conducted through Esakia's embedding theorem of $\mHC$ into $\kfourGrz$. This embedding employs an extension of the G\"{o}del-McKinsey-Tarski translation (mapping $\tr$ below) to modal language with a subsequent splitting (mapping  $\spl$ below), which had for the first time been used in~\cite{mur85} and since then became common place.

First, we expand the language $\Lanm$ by adding another unary modality $\medcircle$ thus obtaining bimodal language $\Lanb$. We denote the set of formulas of the first language by $\Fr_{m}$ and that of the second by $\Fr_{b}$. Next we define the two mappings, $\tr:\Fr_{m}\longrightarrow\Fr_{b}$ and
$\spl:\Fr_{b}\longrightarrow\Fr_{m}$ as follows.
\[
\begin{array}{cl}
\bullet &\tr(p)=\medcircle p~\text{if}~p\in\Var;\\
\bullet &\tr(\alpha\wedge\beta)=\tr(\alpha)\wedge\tr(\beta);\\
\bullet &\tr(\alpha\vee\beta)=\tr(\alpha)\vee\tr(\beta);\\
\bullet &\tr(\alpha\rightarrow\beta)=\medcircle(\tr(\alpha)\rightarrow\tr(\beta));\\
\bullet &\tr(\neg\alpha)=\medcircle\neg\tr(\alpha);\\
\bullet &\tr(\square\alpha)=\medcircle\square\tr(\alpha).
\end{array}
\]
\[
\begin{array}{cl}
\bullet &\spl(p)=p~\text{if}~p\in\Var;\\
\bullet &\spl~\text{commutes with the connectives of $\Lanm$};\\
\bullet &\spl(\medcircle a)=\spl(a)\wedge\square\spl(a),~\text{where $a\in\Fr_{b}$}.
\end{array}
\]
\begin{prop}[Esakia's embedding theorem~\cite{esa06}, Corollary 21]\label{P:Esakia-embedding}
For any formula $\alpha\in\Fr_{m}$,
\[
\mHC\vdash\alpha\Longleftrightarrow\kfourGrz\vdash\spl\circ\tr(\alpha).
\]
\end{prop}
\begin{prop}\label{P:weakening-2}
The inference rule $\square\alpha/\alpha$  is admissible in $\mHC$ and hence in $\E$.
\end{prop}
\begin{proof}
Let $\mHC\vdash\square\alpha$. Then, by virtue of Proposition~\ref{P:Esakia-embedding}, $\kfourGrz\vdash\spl\circ\tr(\square\alpha)$, that is
$\kfourGrz\vdash\square\tr(\alpha)\wedge\square\square\tr(\alpha)$. The later implies that $\kfourGrz\vdash
\square\tr(\alpha)$ and hence, by virtue of Corollary~\ref{C:weakening-1}, $\kfourGrz\vdash\tr(\alpha)$. Applying Proposition~\ref{P:Esakia-embedding} one more time, we obtain that $\mHC\vdash\alpha$; according to Proposition~\ref{P:semantic-equivalence}, the deducibility $\E\vdash\alpha$ is also true. 
\end{proof}
\begin{prop}\label{P:continuum}
There is a continuum of normal extensions of $\E$ which are closed under the weakening rule.
\end{prop}
\begin{proof}
There is a continuum of normal extensions of $\KM$, including $\KM$ itself, which are closed under the weakening rule; cf.~\cite{mur85}, Theorem 3. Since $\KM$ is a normal extension of $\E$, this property is true for extensions of $\E$ as well.
\end{proof}
\begin{conjecture}\label{conjecture-2}
There is a proper normal extension of $\E$ which is properly included in $\KM$ and in which the weakening rule is admissible.
\end{conjecture}

\section{The inference rule $\square\alpha\ra\alpha/\alpha$}
We note that
\[
\E\subset\KM,
\]
for in any nontrivial $\square$-enhanced Heyting algebra with $\square x=x$ the axioms of $\E$ are valid but the formula $(\square p\ra p)\ra p$ is not.

Also, it is seen that the rule
\[
\frac{\square\alpha\ra\alpha}{\alpha}\tag{\textit{L\"{o}b rule}}
\]
is not just admissible in $\KM$ but derivable in it.
The question arises, whether, by adding the L\"{o}b rule to  $\E$, we receive $\KM$ or not? As we will see below, the former is the case.
\begin{prop}\label{P:equivalences}
For any formula $\alpha$, the following conditions are equivalent.
{\em\[
\begin{array}{cl}
(\text{a}) &\KM\vdash\alpha; \\
(\text{b}) &\E + (\text{L\"{o}b rule})\vdash\alpha;\\
(\text{c}) &\E +\square(\square p\ra p)\ra\square p\vdash\alpha.
\end{array}
\]}
\end{prop}
\begin{proof}
To prove the equivalence of (a) and (b), it suffices to show that
\begin{equation*}\label{E:Lob}
\E_{\textbf{icm}}\vdash\square((\square p\ra p)\ra p)
\ra((\square p\ra p)\ra p).
\end{equation*}

We prove the last deducibility algebraically. For this, we observe that the following is true in any modal Heyting algebra with $x\le\square x$:
\[
\begin{array}{rl}
\square((\square x\ra x)\ra x)\wedge(\square x\ra x)\!\!\!\! &=\square((\square x\ra x)\ra x)\wedge\square(\square x\ra x)\wedge(\square x\ra x)
\\
&=\square(((\square x\ra x)\ra x)\wedge(\square x\ra x))\wedge(\square x\ra x)
\\
&=\square((\square x\ra x)\wedge x)\wedge(\square x\ra x)
\\
&=\square x\wedge(\square x\ra x)\\
&\le x.
\end{array}
\]

Similarly, to prove the equivalence of (b) and (c), we show algebraically that
\[
\E_{\textbf{icm}}\vdash\square(\square(\square p\ra p)\ra\square p)
\ra(\square(\square p\ra p)\ra\square p).
\]
Indeed, in any modal Heyting algebra with $x\le\square x$, we obtain:
\[
\begin{array}{rl}
\square(\square(\square x\ra x)\ra\square x)\wedge\square(\square x\ra x)\!\!\!\! &=\square(\square(\square x\ra x)\ra\square x)\wedge\square\square(\square x\ra x)\wedge\square(\square x\ra x)
\\
&=\square((\square(\square x\ra x)\ra x)\wedge\square(\square x\ra x))\wedge\square(\square x\ra x)
\\
&=\square(\square(\square x\ra x)\wedge\square x)\wedge\square(\square x\ra x)
\\
&=\square\square(\square x\ra x)\wedge\square\square x\wedge\square(\square x\ra x)\\
&\le\square\square x
\wedge\square(\square x\ra x)\\
&\le \square\square x\wedge(\square\square x\ra\square x)\\
&\le \square x.
\end{array}
\]
\end{proof}

\section{Assertoric equipollence of $\E$ and $\Int$}\label{S:equipollent}
In this section we aim to prove Proposition~\ref{P:E-and-Int-equipollent}.
Although this proposition follows from a similar proposition for logic $\KM$, (cf.~\cite{mur16a}, Proposition 4.2) in view of Conjecture~\ref{conjecture-1}, it was desirable to obtain the property in question in a direct way.

Let us denote
\[
P(p,q):=(((p\ra q)\ra p)\ra p)\tag{\textit{Peirce law}}
\]
\begin{lem}\label{L:P(p,q)-imp-q-imp-q}
	{\em $\Int_{\textbf{i}}\vdash (P(p,q)\ra q)\ra q$}.
\end{lem}
\begin{proof}
	Our proof is semantical and uses Kripke semantics for $\Int$; see, e.g., \cite{cz97}.  Assume that in some intuitionistic Kripke model $(W,\le,\models)$ the formula $(P(p,q)\ra q)\ra q$ is refuted at a point $a\in W$. That is, $\models$ forces $P(p,q)\ra q$ to be true at $a$ and $q$ to be false at $a$. But this implies that $P(p,q)$ is also false at $a$.  The latter is only possible when there is a point/world $b\in W$ such $a\le b$, where, that is at $b$, $\models$ forces $(p\ra q)\ra p$ to be true and $p$ to be false. The latter implies that $p\ra q$ is also false at $b$. Thus there is $c\in W$ such that $b\le c$ where $p$ is true and $q$ is false.
	We note that, according to a well-known property, the formula $P(p,q)\ra q$, being true at $a$, is also true at any $x\in W$ such that $a\le x$. Applying this property, we obtain that $P(p,q)\ra q$ is true at $c$. Since $p$ is true at $c$, $P(p,q)$ is true at $c$ as well. The latter implies that $q$ is true at $c$. A contradiction.
	Thus $\Int\vdash(P(p,q)\ra q)\ra q$ and, in view of the separation property for $\Int$, (see, e.g.,~\cite{hor62})
	$\Int_{\textbf{i}}\vdash (P(p,q)\ra q)\ra q$ is true as well.
\end{proof}
\begin{cor}
	{\em $\Int_{\textbf{i}}\vdash (P(p,q)\ra q)\leftrightarrow q$}.
\end{cor}
\begin{lem}\label{L:p-implies-P(q,p)}
	{\em $\Int_{\textbf{i}}\vdash p\ra P(q,p)$}.
\end{lem}
\begin{proof}
	It is obvious that
	\[
	\Int_{\textbf{i}}+\lbrace p, (q\ra p)\ra q\rbrace\Vdash q.
	\]
\end{proof}
\begin{lem}\label{L:negative-occurrence}
	{\em$
		\Int_{\textbf{i}}\vdash (q\ra p)\ra(((r\ra p)\ra r)\ra((r\ra q)\ra r)).
		$}
\end{lem}
\begin{proof}
	We prove that
	\[
	\Int_{\textbf{i}}+\lbrace q\ra p,(r\ra p)\ra r,r\ra q\rbrace\Vdash r.
	\]
	Indeed, we have:
	
	\[
	\begin{array}{cll}
	(1) &r\ra q &(\text{premise})\\
	(2) &q\ra p &(\text{premise})\\
	(3) &(r\ra q)\ra((q\ra p)\ra(r\ra p) &(\text{deducible in}~\Int_{\textbf{i}})\\
	(4) &r\ra p &(\text{from $(1), (2), (3)$ by modus ponens twice})\\
	(5) &(r\ra p)\ra r &(\text{premise})\\
	(6) &r &(\text{from $(4)$ and $(5)$ by modus ponens})
	\end{array}
	\]
\end{proof}
\begin{cor}\label{C:A_alpha}
Given a formula $\alpha$, let $A_{\alpha}$ be the assertoric formula obtained from $\alpha$ by deleting of all occurrences of $\square$ in $\alpha$. Then  $\E\vdash\alpha$ implies $\Int\vdash A_{\alpha}$.
\end{cor}
\begin{lem}\label{L:(13)-inKM}
		The following holds$\,:$
	{\em\begin{align*}
		\Int_{\textbf{i}}\vdash(p\ra r)\ra(((r\ra p)\ra p)\ra((((r\ra q)\ra r)\ra(s\ra r))\\
		\ra((q\ra p)\ra (s\ra r)))).
		\end{align*}
	}
\end{lem}
\begin{proof}
	We prove that
	\begin{align*}
	\Int_{\textbf{i}}+\lbrace p\ra r,(r\ra p)\ra p,
	((r\ra q)\ra r)\ra(s\ra r),\\
	q\ra p\rbrace\Vdash s\ra r.
	\end{align*}
	Indeed, we obtain:
	\[
	\begin{array}{cll}
	(1) &q\ra p &(\text{premise})\\
	(2) &(r\ra p)\ra p &(\text{premise})\\
	(3) &(r\ra q)\ra r &(\text{from (1), (2), and Lemma~\ref{L:negative-occurrence} by modus ponens twice})\\
	(4) & ((r\ra q)\ra r)\ra(s\ra r) &(\text{premise})\\
	(5) &s\ra r &(\text{from (3) and (4) by modus ponens})\\
	\end{array}
	\]
\end{proof}

\begin{defn}[refined derivation]
A derivation in a calculus having substitution rule as a postulated rule of inference is called refined if all substitutions, if any, are applied only to the axioms occurring in the derivation and/or to premises, if the derivation has any premises.
\end{defn}

We note that, according to~\cite{sob74} (see also~\cite{lam79}), any derivation in each calculus defined above can be made refined.
\begin{lem}\label{L:main}
	Let $D:\E+\lambda\vdash\sigma$ be a refined derivation. Then for the formulas $\square\alpha$, $\square\beta$ and $\square(\alpha\ra\beta)$, which occur in $D$ as antecedents of the axioms
	of {\em (\textbf{m})} or as the consequent of $(\square\alpha\ra\square\beta)$ which is an inference of the first axiom of {\em (\textbf{m})}, there are corresponding formulas $\boxdot\alpha$, $\boxdot\beta$, and $\boxdot(\alpha\ra\beta)$ such that with corresponding replacements the following deducibility holds:
	\begin{equation}\label{E:main-lemma}
	\begin{array}{r}
	\Intm+\lambda[\square\alpha:\boxdot\alpha,
	\dots\square\beta:\boxdot\beta,\dots\square(\alpha\ra\beta):\boxdot(\alpha\ra\beta)\dots]\\
	+\lbrace\boxdot\alpha\ra P(\boxdot\beta,\alpha),\dots\rbrace
	+\lbrace\boxdot(\alpha\ra\beta)\ra P(\boxdot\beta,\alpha\ra\beta),\dots\rbrace\\
	\Vdash\sigma[\square\alpha:\boxdot\alpha,\dots
	\square\beta:\boxdot\beta,\dots
	\square(\alpha\ra\beta):\boxdot(\alpha\ra\beta)\dots].
	\end{array}
	\end{equation}
\end{lem}
\begin{proof}
Suppose
\begin{equation}
D:\gamma_1,\ldots,\gamma_n,
\end{equation}
where $\gamma_n=\sigma$.
	
Let us denote by $M(D)$ the set of $\square$-formulas
which occur in $D$ as antecedents of the axioms
of  (\textbf{m}) or as a consequent $(\square\alpha\ra\square\beta)$ of the first axiom of  (\textbf{m}).

For each $\square\alpha\in M(D)$, let
\begin{equation}\label{E:3rd-m-axiom-instances}
\square\alpha\ra P(\beta_1,\alpha),\dots,
\square\alpha\ra P(\beta_k,\alpha)
\end{equation}
be all instances of the last axiom of (\textbf{m})
occurring in $D$
that start with $\square\alpha$. We define
\[
\boxdot\alpha=\left\{
	\begin{array}{cl}
	\bigwedge_{1\le j\le k} P(\beta_j,\alpha)[\square\alpha:\true]&\text{if \eqref{E:3rd-m-axiom-instances} is not empty}\\
	\true &\text{if \eqref{E:3rd-m-axiom-instances} is empty}.
	\end{array}\right.
\]

It is obvious that, providing that \eqref{E:3rd-m-axiom-instances} is not empty,
\begin{equation}\label{E:one}
\IntmExt{ic}\vdash\boxdot\alpha\ra P(\beta_j,\alpha),
\end{equation}
where $1\le j\le k$. Next we show that
\begin{equation}\label{E:two}
\IntmExt{ic}\vdash(\boxdot\alpha\ra\alpha)\ra\alpha.
\end{equation}
According to Lemma~\ref{L:P(p,q)-imp-q-imp-q},
\[
\begin{array}{l}
\IntmExt{ic}\vdash(P(\beta_1,\alpha)\ra\alpha)\ra\alpha,\\
\IntmExt{ic}\vdash(P(\beta_2,\alpha)\ra\alpha)\ra\alpha,\\
\dotfill \\
\IntmExt{ic}\vdash(P(\beta_k,\alpha)\ra\alpha)\ra\alpha.\\
\end{array}
\]
Then, we obtain:
\[
\begin{array}{l}
\IntmExt{ic}\vdash((P(\beta_1,\alpha)\wedge P(\beta_2,\alpha)\ra\alpha)\ra(P(\beta_1,\alpha)\ra(P(\beta_2,\alpha)\ra\alpha)),\\
\IntmExt{ic}\vdash((P(\beta_1,\alpha)\wedge P(\beta_2,\alpha)\ra\alpha)\ra(P(\beta_1,\alpha)\ra
\alpha),\\
\IntmExt{ic}\vdash((P(\beta_1,\alpha)\wedge P(\beta_2,\alpha)\ra\alpha)\ra\alpha,\\
\dotfill\\
\IntmExt{ic}\vdash(\bigwedge_{1\le j\le k}P(\beta_j,\alpha)\ra\alpha)\ra\alpha.
\end{array}
\]

Also, by virtue of Lemma~\ref{L:p-implies-P(q,p)},
we have:
\begin{equation}\label{E:three}
\IntmExt{ic}\vdash\alpha\ra\boxdot\alpha.
\end{equation}

Next we prove that if both $\square\alpha,\square\beta\in M(D)$, then
\begin{equation}\label{E:four}
\IntmExt{ic}+\boxdot\alpha\ra P(\boxdot\beta,\alpha)\Vdash
(\alpha\ra\beta)\ra
(\boxdot\alpha\ra\boxdot\beta).
\end{equation}

Indeed, according to Lemma~\ref{L:(13)-inKM},
\begin{align*}
\IntmExt{i}\vdash((\beta\ra\boxdot\beta)\ra(((\boxdot\beta\ra\beta)\ra \beta)\\
\ra((((\boxdot\beta\ra\alpha)\ra \boxdot\beta)\ra(\boxdot\alpha\ra\boxdot\beta))
\ra((\alpha\ra\beta)\ra (\boxdot\alpha\ra\boxdot\beta)))). 
\end{align*}
By virtue of \eqref{E:two} and \eqref{E:two},
\begin{align*}
\IntmExt{ic}\vdash(((\boxdot\beta\ra\alpha)\ra \boxdot\beta)\ra(\boxdot\alpha\ra\boxdot\beta))
\ra((\alpha\ra\beta)\ra (\boxdot\alpha\ra\boxdot\beta)), 
\end{align*}
that is
\[
\IntmExt{ic}\vdash(\boxdot\alpha\ra P(\boxdot\beta,\alpha))\ra ((\alpha\ra\beta)\ra (\boxdot\alpha\ra\boxdot\beta)).
\]
The latter is equivalent to
\begin{align*}
\IntmExt{ic}+\boxdot\alpha\ra P(\boxdot\beta,\alpha)
\Vdash(\alpha\ra\beta)\ra (\boxdot\alpha\ra\boxdot\beta).
\end{align*}

Now suppose $\square(\alpha\ra\beta)\in M(D)$ as the antecedent of the first axiom of (\textbf{m}) (in which case also both $\square\alpha,\square\beta\in M(D)$). Then, we prove,
\begin{equation}\label{E:five}
\begin{array}{r}
\IntmExt{ic}+\lbrace\boxdot\alpha\ra P(\boxdot\beta,\alpha),\boxdot(\alpha\ra \beta)\ra P(\boxdot\beta,\alpha\ra\beta)\rbrace\\
\Vdash
\boxdot(\alpha\ra\beta)\ra (\boxdot\alpha\ra\boxdot\beta).
\end{array}
\end{equation}
We will be proving that
\begin{align*}
\IntmExt{ic}+\lbrace\boxdot\alpha\ra P(\boxdot\beta,\alpha),
\boxdot(\alpha\ra \beta)\ra P(\boxdot\beta,\alpha\ra\beta),\\
\boxdot(\alpha\ra\beta)\rbrace\Vdash
\boxdot\alpha\ra\boxdot\beta.
\end{align*}
Indeed, we obtain:
\[
\begin{array}{cll}
(1) & \boxdot(\alpha\ra\beta) &(\text{premise})\\
(2) & \boxdot(\alpha\ra \beta)\ra P(\boxdot\beta,\alpha\ra\beta) &(\text{premise})\\
(3) & (((\boxdot\beta\ra(\alpha\ra\beta))\ra\boxdot\beta)\ra\boxdot\beta) &(\text{from (1) and (2) by modus ponens})\\
(4) &(\boxdot\beta\ra(\alpha\ra\beta))\ra(\alpha\ra
(\boxdot\beta\ra\beta))
& (\text{deducible in $\IntmExt{i}$})\\
(5) & (\boxdot\beta\ra\beta)\ra\beta &(\text{by virtue of~\eqref{E:two}, since $\square\beta\in M(D)$})\\
(6) & (\boxdot\beta\ra(\alpha\ra\beta))\ra(\alpha\ra \beta)
&(\text{from (4) and (5) deducible in $\IntmExt{i}$})\\
(7) &\boxdot\alpha\ra P(\boxdot\beta,\alpha)
&(\text{premise})\\
(8) & (\alpha\ra\beta)\ra(\boxdot\alpha\ra\boxdot\beta) 
&(\text{by virtue of~\eqref{E:four}})\\
(9) & (\boxdot\beta\ra(\alpha\ra\beta))\ra(\boxdot\alpha\ra
\boxdot\beta)
&(\text{from (6) and (8) deducible in $\Intm_{\textbf{i}}$})\\
(10) & \boxdot\alpha\ra ((\boxdot\beta\ra(\alpha\ra \beta))\ra\boxdot\beta)
&(\text{from (9) deducible in $\Intm_{\textbf{i}}$})\\
(11) &\boxdot\alpha\ra\boxdot\beta  &(\text{from (3) and (10) deducible in $\Intm_{\textbf{i}}$})
\end{array}
\]

Now for each $i$, $1\le i\le n$, we define:
\[
\gamma_{i}^{\ast}:=\gamma_i[\alpha:\boxdot\alpha,\dots,\square\beta:\boxdot\beta,\dots,\square(\alpha\ra\beta):\boxdot(\alpha\ra\beta),\dots]
\]
and further
\[
[\gamma_{i}^{\ast}]:=\left\{
\begin{array}{cl}
\gamma_{i}^{\ast} &\text{if $\gamma_i$ is not an instance of an \textbf{m}-axiom}\\
\text{a $\IntmExt{ic}$-derivation according to \eqref{E:one}} &\text{if $\gamma_i$ is an instance of the third \textbf{m}-axiom}\\
\text{a $\IntmExt{ic}$-derivation according to \eqref{E:three}} &\text{if $\gamma_i$ is an instance of the second \textbf{m}-axiom}\\
\text{a $\IntmExt{ic}$-derivation according to \eqref{E:five}} &\text{if $\gamma_i$ is an instance of the first \textbf{m}-axiom}.\\
\end{array}
\right.
\]

Now let us consider
\[
D^{\ast}:[\gamma_{1}^{\ast}],\ldots,[\gamma_{n}^{\ast}].
\]
One can see that $D^{\ast}$ supports the deducibility \eqref{E:main-lemma}.
\end{proof}
\begin{prop}\label{P:E-and-Int-equipollent}
The calculi $\E$ and $\Int$ are assertorically equipollent; that is for any assertoric formulas $A$ and $B$,
\[
\E+A\vdash B\Longleftrightarrow\Int+A\vdash B.
\]
\end{prop}
\begin{proof}
Suppose $D:\E+\lambda\vdash B$ is a refined derivation, where $\lambda$ is the conjunction of instances of $A$. According to Lemma~\ref{L:main}, 
\begin{equation}\label{E:last}
\begin{array}{r}
D^{\ast}:\Intm+\lambda[\square\alpha:\boxdot\alpha,
\dots\square\beta:\boxdot\beta,\dots\square(\alpha\ra\beta):\boxdot(\alpha\ra\beta)\dots]\\
+\lbrace\boxdot\alpha\ra P(\boxdot\beta,\alpha),\dots\rbrace
+\lbrace\boxdot(\alpha\ra\beta)\ra P(\boxdot\beta,\alpha\ra\beta),\dots\rbrace\\
\Vdash B,
\end{array}
\end{equation}
for some formulas $\boxdot\alpha,\dots\boxdot\beta\dots\boxdot(\alpha\ra\beta),\dots$

We denote
\[
\lambda^{\ast}:=\lambda[\square\alpha:\boxdot\alpha,
\dots\square\beta:\boxdot\beta,\dots\square(\alpha\ra\beta):\boxdot(\alpha\ra\beta)\dots].
\]
Now let
\begin{equation}\label{E:last+1}
\boxdot\gamma\ra P(\beta_1,\gamma),\ldots,
\boxdot\gamma\ra P(\beta_m,\gamma)
\end{equation}
be all formulas of the second row of \eqref{E:last} which begin with $\boxdot\gamma$. We note that each $\beta_j$ does not contain $\boxdot\gamma$. Next, for each such a formula $\boxdot\gamma$, we define:
\[
\boxtimes\gamma:=\bigwedge_{1\le j\le m}P(\beta_j,\gamma).
\]
We observe that
\begin{equation}\label{E:last+2}
\IntmExt{ic}\vdash\boxdot\gamma\ra P(\beta_1,\gamma)[\boxdot\gamma:\boxtimes\gamma,\dots]
\end{equation}
Let
\[
D^{\ast}: \gamma_1,\ldots,\gamma_n.
\]
Then, we define
\[
\gamma^{\ast}_{i}:=\gamma_i[\boxdot\gamma:\boxtimes\gamma,\dots]
\]
and
\[
[\gamma^{\ast}_{i}]=\left\{
\begin{array}{cl}
\gamma^{\ast}_{i} &\text{if $\gamma_{i}$ is not one of~\eqref{E:last+1}}\\
\text{a derivation supported by~\eqref{E:last+2}}
&\text{if $\gamma_{i}$ is one of~\eqref{E:last+1}}.
\end{array}
\right.
\]
Denoting
\[
D^{\ast\ast}: [\gamma^{\ast}_{1}],\ldots,[\gamma^{\ast}_{1}],
\]
we observe that $\Intm+\lambda^{\ast\ast}\Vdash B$, where $\lambda^{\ast\ast}:=\lambda^{\ast}[\boxdot\gamma:\boxtimes\gamma,\dots]$. Thus we have $\Intm\vdash\lambda^{\ast}\ra B$. By virtue of Corollary~\ref{C:A_alpha}, $\Int\vdash A_{\lambda^{\ast\ast}\rightarrow B}$. We note that
$A_{\lambda^{\ast\ast}\rightarrow B}=C\ra  B$, where $C$ is conjunction of instances of $A$. Hence $\Int+A\vdash B$.
\end{proof}
\begin{cor}
Any assertoric formula is derivable in $\E$ if it is also derivable by using only axioms of the group {\em 
	(\textbf{i})} and those ones in the groups {\em 
	$(\textbf{c})-(\textbf{n})$} which correspond to the logical connectives actually appearing in the formula
\end{cor}


\end{document}